\documentclass[10pt,leqno, english]{amsart}
\usepackage[colorlinks=true]{hyperref}
\usepackage{srcltx}
\usepackage[english]{babel}
\usepackage{amsmath, amsthm}
\numberwithin{equation}{section}
\usepackage[top=3.5cm, bottom=3.5cm, left=2cm, right=2cm,marginparwidth=1.8cm, marginparsep=0.1cm]{geometry}

\usepackage{amsfonts}
\usepackage{amssymb}
\usepackage{bbm}
\usepackage{mathrsfs}
\usepackage{tikz-cd}
\usepackage{tensor}
\usepackage{mathtools}
\usepackage{MnSymbol}
\usepackage{geometry}
\usepackage{marginnote}

\newtheorem{proposition}{Proposition}[section]
\newtheorem{theorem}[proposition]{Theorem}

\newtheorem{lemma}[proposition]{Lemma}
\newtheorem{corollary}[proposition]{Corollary}

\newtheorem{claim}[proposition]{Claim}
\newtheorem*{problem*}{Inverse Sieve Problem}

\theoremstyle{definition}

\theoremstyle{remark}
\newtheorem{remark}[proposition]{Remark}

\newtheorem*{remark*}{Remark}
\newtheorem*{remarks*}{Remarks}

\def\house#1{{%
    \setbox0=\hbox{$#1$}
    \vrule height \dimexpr\ht0+1.4pt width .5pt depth \dp0\relax
    \vrule height \dimexpr\ht0+1.4pt width \dimexpr\wd0+2pt depth \dimexpr-\ht0-1pt\relax
    \llap{$#1$\kern1pt}
    \vrule height \dimexpr\ht0+1.4pt width .5pt depth \dp0\relax}}

\begin{document}

\title[]{Effective Hilbert's Irreducibility Theorem for global fields}

\author[M. Paredes, R. Sasyk]{Marcelo Paredes $^{2}$ \MakeLowercase{and} Rom\'an Sasyk $^{1,2}$}

\address{$^{1}$Instituto Argentino de Matem\'aticas Alberto P. Calder\'on-CONICET,
Saavedra 15, Piso 3 (1083), Buenos Aires, Argentina;}

\address{$^{2}$Departamento de Matem\'atica, Facultad de Ciencias Exactas y Naturales, Universidad de Buenos Aires, Argentina.}

\email{\textcolor[rgb]{0.00,0.00,0.84}{mparedes@dm.uba.ar}}
\email{\textcolor[rgb]{0.00,0.00,0.84}{rsasyk@dm.uba.ar}}



\begin{abstract}	
We prove an effective form of  Hilbert's irreducibility theorem for polynomials over a global field $K$. More precisely, we give effective bounds for the number of specializations $t\in \mathcal{O}_K$ that do not preserve the irreducibility or the Galois group of a given irreducible polynomial $F(T,Y)\in K[T,Y]$. The bounds are explicit in the height and degree of the polynomial $F(T,Y)$, and are optimal in terms of the size of the parameter $t\in \mathcal{O}_K$. Our proofs deal with the function field and number field cases in a unified way. 
\end{abstract}

\maketitle

\section{Introduction}

A cornerstone in diophantine geometry and Galois theory is Hilbert's irreducibility theorem, which in its original form proved in \cite{Hilbert1892} states that for any polynomial $F(T,Y)\in \mathbb{Z}[T,Y]$ of degree at least one in $Y$, irreducible in $\mathbb{Q}[T,Y]$ there exist infinitely many integers $n\in \mathbb{Z}$ such that the specialized polynomial $F(n,Y)\in \mathbb{Z}[Y]$ is irreducible in $\mathbb{Q}[Y]$. This was then used by Hilbert to study the inverse Galois problem, more precisely, he showed that in order to prove that a finite group $G$ is the Galois group of a finite extension of $\mathbb{Q}$ it is enough to show that $G$ is the Galois group of some finite extension of the function field $\mathbb{Q}(T)$. Furthermore he constructed polynomials with rational coefficients with Galois group $S_{n}$ and $A_{n}$ for all $n$. 

Since then, there have been numerous proofs, improvements and generalizations of Hilbert's irreducibility theorem. Of special interest are those variants which are quantitative. For instance, let $K$ be a global field and let $\mathcal{O}_K$ be its ring of integers. Let us consider a polynomial $F\in \mathcal{O}_K[T_1,\ldots, T_s,Y_1,\ldots, Y_r]$ which is irreducible as a polynomial in $K(T_1,\ldots, T_s)[Y_1,\ldots, Y_r]$. We want to mesure   how likely is it that a specialization $F(\boldsymbol t,Y_1,\ldots, Y_r)$ with $\boldsymbol t:=(t_1,\ldots ,t_s)\in \mathcal{O}_K^s$  remains irreducible. In the case when $K=\mathbb{Q}$ it is natural to consider the quantity
\[N_{F,\mathbb{Q}}(B)  :=\left|\left\{ \boldsymbol t\in \mathbb{Z}^{s}:\max_i|t_i|\leq B\text{ and }F(\boldsymbol t,X_{1},\ldots, X_{r})\text{ is not irreducible in }\mathbb{Q}[X_{1},\ldots, X_{r}] \right\}\right|,\]
whereas for global fields a natural generalization of this quantity is 
\[N_{F,K}(B):=\left|\left\{ \boldsymbol t\in [B]_{\mathcal{O}_K}^s:F(\boldsymbol t,Y_{1},\ldots, Y_{r})\text{ is not irreducible in }K[Y_{1},\ldots, Y_{r}] \right\}\right|,\]
where $[B]_{\mathcal{O}_K}^s$ is the global field analogue of the integer box $[-B,B]^s \cap \mathbb{Z}$, which will be defined in Section \ref{alturas}.

We also want to mesure how likely is that a specialization of $F$ on $\boldsymbol t$ preserves the Galois group of $F$. Specifically, for a global field $K$ and $F\in \mathcal{O}_{K}[T_1,\ldots, T_s,Y]$ an irreducible polynomial, let $L$ be the splitting field of $F$ over $K(T_1,\ldots, T_s)$. Let $G$ be the Galois group of $L/K(T_1,\ldots, T_s)$. For a given $\boldsymbol t\in \mathcal{O}_K^s$ let $L_{\boldsymbol t}$ be the splitting field of $F(t_1,\ldots, t_s, Y)$, and let $G_{\boldsymbol t}$ be the Galois group of $L_{\boldsymbol t}/K$. Then one wants to estimate the quantity
 \[E_{F,K}(B):=\left|\left\{\boldsymbol t\in [B]_{\mathcal{O}_K}^s:G_{\boldsymbol t}\neq G\right\}\right|.\]
 
We will use the asymptotic notation $X\lesssim_{C_1,\ldots,C_n }Y$ to mean $|X|\leq c|Y|$ for some constant $c$ depending on the parameters $C_1,\ldots ,C_n$. With this notation at hand, we record some of the known estimates for $|E_{F,K}(B)|$. When $K$ is a number field,  in \cite[Theorem 2.1 and Theorem 2.5]{Cohen} Cohen used the large sieve to prove that the quantities $N_{F,K}(B)$ and $E_{F,K}(B)$ are both bounded by $\lesssim_{r,s,K,F}B^{s-\frac{1}{2}}\log(B)$ with the implicit constant depending polynomially on the height of the polynomial $F$. Furthermore in \cite[Section 13, Theorem 1]{Serre} Serre showed that the bounds can be improved to $\lesssim_{r,s,K,F}B^{s-\frac{1}{2}}(\log(B))^{\gamma}$ with $\gamma<1$. In the case when $K=\mathbb{Q}$ and $s=r=1$, in \cite{SZ} by means of the determinant method of Bombieri and Pila \cite{Bombieri0}, Schinzel and Zannier obtained effective estimates for $N_{F,\mathbb{Q}}(B)$. Improving upon this work, in \cite[$\S$ 3.2]{Walkowiak} Walkowiak used the $p$-adic determinant method of Heath-Brown \cite{Heath-Brown} to prove the bound $N_{F,\mathbb{Q}}(B)\leq c(\log(H(F)))^{19}B^{\frac{1}{2}}(\log(B))^5$ with an explicit $c$ depending on the degree of $F$, and with $H(F)$ the height of $F$. Recently, when $K=\mathbb{F}_q(T)$, \cite[Theorem 1.1]{BSE} Bary-Soroker and Entin used the large sieve for function fields to prove that Cohen's bound also holds in this case.   

One may further sharpen the counting function $E_{F,K}(B)$ to take into consideration the algebraic structure of $G$. In this direction, in \cite[Theorem 1.4]{Zywina} Zywina used the larger sieve of Gallagher to prove that for any number field $K$, if $L/K(T_1,\ldots, T_s)$ is geometric (i.e. if $L\cap \overline{K}=K$) it holds the bound $E_{F,K}(B)\lesssim_{r,s,K,F} B^{s-1+\beta_G}\log(B)$ where $\beta_G:=\max_{M\leq G}\frac{|\bigcup_{g\in G}gMg^{-1}|}{|G|}$ and $M$ runs over the maximal subgroups of $G$. In \cite[Corollary 1]{Castillo}, Castillo and Dietman used Galois resolvents, and bounds obtained in \cite{Browning0} by means of the $p$-adic determinant method, to show that for $K=\mathbb{Q}$, for any subgroup $H$ of $G$, and for all $\varepsilon>0$ it holds that
\[\left|\left\{\boldsymbol t\in [B]_{\mathcal{O}_K}^s:G_{\boldsymbol t}=H\right\}\right|\lesssim_{F,\varepsilon} B^{s-1+|G/H|^{-1}+\varepsilon}.\]
Among other applications, the p-adic determinant method was used in \cite{Salberger2, Heath-Brown} by Browning, Heath-Brown, and Salberger  to prove the uniform dimension growth conjecture of Heath-Brown and Serre for varieties over $\mathbb{Q}$ of degree $d\geq 6$. Specifically, they proved that for any integral projective variety $X$ defined over $\mathbb{Q}$ of degree $d\geq 6$, it holds
\[|\{\boldsymbol x\in X(\mathbb{Q}):H(\boldsymbol x)\leq B\}|\lesssim_{\dim(X),d,\varepsilon}B^{\dim(X)+\varepsilon},\]
where $H(\boldsymbol x)$ is the projective height of $\boldsymbol x$. The conjecture for degree $d\geq 4$ was solved by Salberger in \cite{Salberger}. In order to do so, in that article Salberger introduced the global determinant method. This method was refined by Walsh in \cite{Walsh3} in his solution of a conjecture of Heath-Brown posed in \cite{Heath-Brown} on bounds for rational points of integral curves.  All this was further explored in \cite{Cluckers} by Castrick, Cluckers, Dittmann, and Nguyen to prove that the dimension growth conjecture holds without the $\varepsilon$ factor when $d\geq 5$ and with a polynomial dependency in $d$. In \cite{Vermeulen} Vermeulen adapted \cite{Cluckers} to prove the uniform dimension growth conjecture for integral projective varieties defined over $\mathbb{F}_q(T)$ of degree $d\geq 64$. Finally, in \cite{PS} we extended \cite{Cluckers, Salberger, Vermeulen} and proved the uniform dimension growth conjecture for integral projective varieties over global fields of degree $d\geq 4$.

In the present article we apply some techniques and results of \cite{PS} to extend and improve the aforementioned results on effective Hilbert's irreducibility theorem to global fields. Specifically, in \cite{PS} we proved a Bombieri-Pila  type of bound for curves over global fields, which depends polynomially  on the degree of the curve (the cases $K=\mathbb{Q}$ and $K=\mathbb{F}_q(T)$ were already covered in \cite{Cluckers,Vermeulen}). This, together with a technique we elaborated in \cite{PS} that allows us to suppose that a polynomial with coefficients in a global field has coefficients in $\mathcal{O}_K$ and it is almost primitive in the case when $\mathcal{O}_K$ is not a principal ideal domain, are the main novelties in our approach. It seems pertinent to emphasize that our proofs deal simultaneously with the number field and the function field cases. Because of that, for any $a_{1},a_{2}\in \mathbb{R}$ we will use the notation $[a_{1},a_{2}]:=\begin{cases}a_{1} & \text{ if char}(K)=0\\ a_{2} & \text{ if char}(K)>0\end{cases}$. We are now in position to state the first result of this article.

\begin{theorem}
Let $K$ be a global field, and let $F\in K[T,Y]$ be an irreducible polynomial of degree $d_{T}$ in $T$ and $d_Y$ in $Y$. Then 
\[N_{F,K}(B)\lesssim_{K} 2^{[25d_{Y},36d_{Y}]}d_{Y}^{[26,35]}d_{T}^{[21,26]}\left(\log(H_{K}(F))+1\right)^{[6,10]}B^{\frac{1}{2}},\]
where $H_K(F)$ denotes the $K$-relative height of $F$ defined in Section \ref{alturas}.
\label{HIT01}
\end{theorem}
Theorem \ref{HIT01} gives the best possible bound in terms of the magnitude of $B$, as it can be seen by the example $F(T,Y)=Y^2-T^2$. Furthermore the dependence on $F$ in the bound is completely explicit. When $K=\mathbb{Q}$ this result improves upon the bounds given in \cite{Walkowiak}. We remark that  for number fields $K\neq \mathbb{Q}$ and for function fields this is the first effective bound of this type (the non effective bound $N_{F,K}(B)\lesssim_{F,K}B^{\frac{1}{2}}$ for number fields is a consequence of Siegel's finiteness theorem). 

In order to prove Theorem \ref{HIT01} we follow the classical approach that reduces the problem of bounding $N_{F,K}(B)$ to that of bounding the number of $\mathcal{O}_K$-roots of any specialization of a certain irreducible polynomial $P(T,Y)\in \mathcal{O}_K[T,Y]$. In that sense, we prove the following result.

\begin{theorem}
Let $K$ be a global field, and let $P\in \mathcal{O}_{K}[T,Y]$ be an irreducible polynomial. Let $d_{T}$ and $d_{Y}$ be the degrees in $T$ and $Y$ of $P$, respectively, and let $d$ be the total degree of $P$. Then the number of $t\in [B]_{\mathcal{O}_{K}}$ such that $P(t,Y)$ has a solution $y\in \mathcal{O}_{K}$ is
\[\left|\left\{t\in [B]_{\mathcal{O}_{K}}:P(t,Y)\text{ has a solution }y\in \mathcal{O}_{K}\right\}\right|\lesssim_{K} d_{Y}^{[12,16]}d_{T}^{[5,9]}(\log(H_{K}(P))+1)^{[6,10]}B^{\frac{1}{d_{Y}}}. \] 
\label{Hilbert1}
\end{theorem}
 
For number fields, Theorem \ref{Hilbert1} improves upon \cite[$\S$ 2.2]{Walkowiak}  and \cite[Corollary C]{Motte} by removing the $\log(B)$ factor, and lowering the exponents in the degree and the height of $P$. For function fields, Theorem \ref{Hilbert1} seems to be new, even for $K=\mathbb{F}_q(T)$. As a consequence of Theorem \ref{Hilbert1}, we give the following bound for $E_{F,K}(B)$.

\begin{theorem}
Let $K$ be a global field and let $F\in K[T,Y]$ be an irreducible  polynomial of degrees $d_T$ in $T$ and $d_Y$ in $Y$. Let $G$ be the Galois group of $F$ and let $\mathcal{H}$ be the family of subgroups of $G$. Let us suppose that $F$ is separable and monic as a polynomial in $K(T)[Y]$. Then
\begin{align*}
E_{F,K}(B) & \lesssim_{K}\sum_{H\in \mathcal{H}}2^{[25|G/H|,36|G/H|]}|G/H|^{[26,35]}d_T^{[28,37]}d_Y^{[7,11]}|H|^{[27,36]}|G|^{[28,37]}(\log(H_{K}(F))+1)^{[6,10]}B^{\frac{1}{2}} \\ & \lesssim_{K,d_T,d_Y} (\log(H_{K}(F))+1)^{[6,10]}B^{\frac{1}{2}}. 
\end{align*}
Moreover, in the case when $F$ is non monic, the second inequality still holds true.
\label{Hilbert3.5}
\end{theorem}

In the case of number fields, Theorem \ref{Hilbert3.5} improves upon the result of \cite{Cohen} when $s=1$, whereas, in the case where $K=\mathbb{F}_q(T)$, Theorem \ref{Hilbert3.5} improves the result of \cite{BSE} when $s=1$ by removing the $\log(B)$ factor, and making explicit the dependence on the degree and the height of $F$. Once again, the novelty in the number field case lies in the effective dependence on the height of $F$, since  it was shown in \cite[$\S$ 9.2, Proposition 2]{Serre} that the set of specializations for which $G_{t}\neq G$ is an affine thin subset, and then by means of Siegel's finiteness theorem, in \cite[$\S$ 9.7]{Serre} it was shown that for number fields the number of $t\in [B]_{\mathcal{O}_K}$ lying in an affine thin subset $\Omega$ is $\lesssim_{K,\Omega}B^{\frac{1}{2}}$.

While the order of magnitude on $B$ of the bound of Theorem \ref{Hilbert3.5} is in general sharp as can be seen by considering the polynomial $F(T,Y)=Y^2-T$, in some cases one may obtain better bounds. Indeed, when $K=\mathbb{Q}$ in \cite{Castillo} Castillo and Dietmann used Galois resolvents to find an adequate primitive element for the splitting field $L$ of $F$, and use \cite[Theorem 1]{Browning0} (which is a Bombieri-Pila type of bound over $\mathbb{Q}$ for lopsided boxes) to deduce a variant of Hilbert's irreducibility theorem for Galois groups that takes into consideration the subgroup structure of the Galois group of the specialized polynomial. By means of Theorem \ref{Hilbert1} and a lemma about Galois resolvents in \cite{Castillo}, we improve and generalize \cite[Theorem 1]{Castillo} to number fields, and we improve upon Theorem \ref{Hilbert3.5} in the case when $K$ is a number field. Our result reads as follows.
  
\begin{theorem}
Let $K$ be a number field. Let $F(T,Y)\in \mathcal{O}_{K}[T,Y]$ be an irreducible polynomial of degrees $d_T$ in $T$ and $d_Y$ in $Y$. Let $G$ be the Galois group of $F$ over $K(T)$, and let $H$ be a subgroup of $G$. Then
\[\left|\left\{ t\in [B]_{\mathcal{O}_K}: \text{ the splitting field of }F(t,Y)\text{ over }K\text{ has Galois group }H \right\}\right|\lesssim_{K,d_Y,d_T}(\log(H_{K}(F))+1)^{[6,10]}B^{|G/H|^{-1}}.\]
\label{Hilbert7}
\end{theorem}
As a corollary of Theorem \ref{Hilbert7} we also obtain the following result.

\begin{theorem}
Let $K$ be a number field, and let $F(T,Y)\in \mathcal{O}_{K}[T,Y]$ be an irreducible polynomial of degrees $d_T$ in $T$ and $d_Y$ in $Y$. Let $G$ be the Galois group of $F(T,Y)$ over $K(T)$. If $\delta_{G}:=\max\{|G/H|^{-1}:H\text{ is a proper subgroup of }G\}$ and $\gamma_{G}:=\max\{|G/H|^{-1}:H\text{ is an intransitive subgroup of }G\}$, then
\[E_{F,K}(B)\lesssim_{K,d_Y,d_T}(\log(H_{K}(F))+1)^{[6,10]}B^{\delta_G},\]
and
\[N_{F,K}(B)\lesssim_{K,d_Y,d_T}(\log(H_{K}(F))+1)^{[6,10]}B^{\gamma_G}.\]
\label{HIT3}
\end{theorem}
This bound is new in the case when $K\neq \mathbb{Q}$, whereas when $K=\mathbb{Q}$, this improves the bound given  in \cite{Castillo}. 

Before ending this introduction, we must mention that the dependence on $K$ of the implicit constants when $K$ is a  function field can be made explicit in the degree and genus of the field in all of the bounds, whereas for number fields this dependence is explicit in the degree and the discriminant only if one assumes the Generalized Riemann Hypothesis. This is because the dependence on $K$ in this article relies on the dependence on $K$ in the bounds obtained in \cite{PS}, which in turn depend on estimates for the number of primes ideals of bounded norm. In the function field case this is covered by the Riemann Hypothesis for curves over finite fields, while in the number field case this is covered by the Landau prime ideal theorem, for which explicit versions are only known under the Generalized Riemann Hypothesis (see \cite[Remark 2.7]{PS}).  

\vskip0.5cm

\noindent\textbf{Acknowledgments.} M. Paredes was supported in part by a CONICET Postdoctoral Fellowship. R. Sasyk was supported in part through the grant PICT 2017-0883. We thank the anonymous referee for the careful reading of the manuscript and his or her many suggestions.

\section{Heights in global fields}
\label{alturas}

We use the asymptotic notation $X=O(Y)$ or $X\lesssim Y$ to mean $|X|\leq C|Y|$ for some constant $C$. We also use $O_{K,n,d}(Y)$ or $\lesssim_{K,n,d}Y$ to mean that the implicit constants depend on $K,n$ and $d$.

Throughout this paper, $K$ denotes a global field, i.e. a finite separable extension of $\mathbb Q$ or $\mathbb{F}_{q}(T)$, in  which case we further assume that the field of constants is $\mathbb{F}_{q}$. We will denote by $d_{K}$ the degree of the extension  $K/ \mathbbm{k}$,  where $\mathbbm{k}$ indistinctively denotes the base fields  $\mathbb{Q}$  or  $\mathbb{F}_{q}(T)$.

Let $K$ be a number field and let $\mathcal{O}_{K}$ be its ring of integers. Then each embedding $\sigma:K\hookrightarrow \mathbb{C}$ induces a place $v$, by means of the equation
\begin{equation}
|x|_{v}:=|\sigma(x)|^{\frac{n_{v}}{d_{K}}}_{\infty},
\label{definition of arq places}
\end{equation}
where $|\cdot |_{\infty}$ denotes the absolute value of $\mathbb{R}$ or $\mathbb{C}$ and $n_{v}=1$ or $2$, respectively. 
Such places will be called the places at infinity, and will be denoted by $M_{K,\infty}$. Note that $\sum_{v\in M_{K,\infty}}n_{v}=d_{K}$. They are all the archimedean places of $K$. Since the complex embeddings come in pairs that differ by complex conjugation, we have $|M_{K,\infty}|\leq d_{K}$.

Now, let us suppose that $K$ is a function field over $\mathbb{F}_{q}$, such that $\mathbb{F}_{q}$ is algebraically closed in $K$ (in other words, the constant field of $K$ is $\mathbb{F}_{q}$). A prime in $K$ is, by definition, a discrete valuation ring $\mathcal{O}_{(\mathfrak{p})}$ with maximal ideal $\mathfrak{p}$ such that $\mathbb{F}_{q}\subseteq \mathfrak{p}$ and the quotient field of $\mathcal{O}_{(\mathfrak{p})}$ equals to $K$. By abuse of notation, when we refer to a prime in $K$, we will refer to the maximal ideal $\mathfrak{p}$. Associated to $\mathfrak{p}$, we have the usual $\mathfrak{p}$-adic valuation, that we will denote by $\text{ord}_{\mathfrak{p}}$. The degree of $\mathfrak{p}$, denoted by $\deg(\mathfrak{p})$ will be the dimension of $\mathcal{O}_{(\mathfrak{p})}/\mathfrak{p}$ as an $\mathbb{F}_{q}$-vector space, which is finite. Then the norm of $\mathfrak{p}$ is defined as $\mathcal{N}_{K}(\mathfrak{p}):=q^{\deg(\mathfrak{p})}$. 
Any prime $\mathfrak{p}$ of $K$ induces a place $v$ in $K$ by the equation
\[|x|_{v}:=|x|_{\mathfrak{p}}:=\mathcal{N}_{K}(\mathfrak{p})^{-\frac{\text{ord}_{\mathfrak{p}}(x)}{d_{K}}}.\]
They are all the places in $K$. The set of all places in $K$ is denoted by $M_{K}$. Now we fix an arbitrary place $v_{\infty}$ in $M_{K}$ above the place in $\mathbb{F}_{q}(T)$ defined by $\left |\frac{f}{g}\right |_{\infty}:=q^{\deg(f)-\deg(g)}$. Its corresponding prime will be denoted $\mathfrak{p}_{\infty}$; it has degree at most $d_{K}$. The ring of integers of $K$ is the subset
\[\mathcal{O}_{K}:=\left\{ x\in K: |x|_{v}\leq 1\text{ for all }v\in M_{K},v\neq  v_{\infty}\right\}.\]

If $K$ is a number field with ring of integers $\mathcal{O}_{K}$, we define
\[[B]_{\mathcal{O}_{K}}:=\{x\in \mathcal{O}_{K}:\max_{\sigma:K\hookrightarrow \mathbb{C}}|\sigma(x)|_{\infty}\leq B^{\frac{1}{d_{K}}}\}=\{x\in \mathcal{O}_{K}:\max_{v\in M_{K,\infty}}|x|_{v}\leq B^{\frac{n_v}{d_{K}^2}}\}.\]
When $K$ is a function field, we define
\[[B]_{\mathcal{O}_{K}}:=\{x\in \mathcal{O}_{K}:|x|_{v_{\infty}}\leq B^{\frac{1}{d_{K}}}\}.\]

We will use a notion of height for a polynomial that is defined in \cite[$\S$ 1.6]{Bombieri}. Specifically, given a global field $K$ of degree $d_{K}$ and a place $v\in M_{K}$, for $F=\sum_{I}a_{I}\boldsymbol Y^{I}\in K[Y_{1},\ldots, Y_{n}]$, we define
\[|F|_{v}:=\max_{I} |a_{I}|_{v}.\]
Following \cite[$\S$ 1.6]{Bombieri}, the $K$-relative height of $F\in K[Y_{1},\ldots ,Y_{n}]$ is defined as 
\[H_{K}(F):=\left(\prod_{v\in M_{K}}|F|_{v}\right)^{d_{K}}.\]
We will also use the $K$-affine relative height
\[H_{K,\text{aff}}(F):=\left(\prod_{v\in M_{K}}\max\{1,|F|_{v}\}\right)^{d_{K}}.\]
In particular,
\[H_{K}(F)\leq H_{K,\text{aff}}(F).\]

We will use the following property concerning the height of polynomials, which is a consequence of \cite[Proposition 2.2]{PS}, that will allow us to deal with polynomials defined in  global fields with non-principal ring of integers.
\begin{lemma}
Let $K$ be a global field. There exists a positive constant $c_{1}=c_{1}(K)$ with the property that for any $F\in K[Y_{1},\ldots, Y_{n}]$ there is $\lambda\in K^{\times}$ such that $\lambda F\in \mathcal{O}_{K}[Y_{1},\ldots, Y_{n}]$, and for all $v\in M_{K,\infty}$ it holds  
\begin{equation}
|\lambda F|_{v}\leq \begin{cases}c_{1}H_{K}(F)^{\frac{n_{v}}{d_{K}^{2}}} & \text{ if }K \text{ is a number field}, \\ c_{1}H_{K}(F)^{\frac{1}{d_{K}}} & \text{ if }K \text{ is a function field}.\end{cases}
\label{serrebound}
\end{equation}
In particular, $H_{K,\emph{\text{aff}}}(\lambda F)\lesssim_K H_K(F)$.
\label{serre}
\end{lemma}

We will use the following standard estimate of Liouville (or perhaps  Cauchy). Let $K$ be a field with  an absolute value $|\cdot|$. Let $P=a_0+a_1Y+\cdots +a_dY^d\in K[Y]$. of degree $d\geq 0$. Let $y$ be a root of $P$ over an algebraic closure $\overline{K}$ of $K$. Let us suppose that $|\cdot |$ also denotes an extension of the absolute value to $K(y)$. Then
\begin{equation}
|y|\leq \begin{cases}1+\frac{\max_i|a_i|}{|a_d|} & \text{ if }|\cdot |\text{ is archimedean}\\ \frac{\max_i|a_i|}{|a_d|} & \text{ if }|\cdot | \text{ is non-archimedean}\end{cases}.
\label{liouville}
\end{equation}
Indeed, If $|\cdot|$ is archimedean, then for any $|y|\geq 1$ it holds the inequality
\[|a_d ||y|^d=\left|\sum_{i=0}^{d-1}a_i y^i\right|\leq \sum_{i=0}^{d-1}|a_i||y|^i\leq \max|a_i|\sum_{i=0}^{d-1}|y|^i=\max|a_i| \frac{|y|^d-1}{|y|-1}\leq \max|a_i|\frac{|y|^d}{|y|-1},\]
from where the bound follows at once. On the other hand, If $|\cdot |$ is non-archimedean,  \eqref{liouville} follows by replacing the triangle inequality by the ultrametric inequality.

\section{Bounds for the number of integral roots of the specialized polynomial}
\label{section 3}

As with many proofs of Hilbert's irreducibility theorem,  we are going to reduce the problem to that of bounding the number of $t\in [B]_{\mathcal{O}_{K}}$  for which $P(t,Y)$ has an $\mathcal{O}_{K}$-root, with $P(T,Y)\in \mathcal{O}_{K}[T,Y]$ an adequate irreducible polynomial of degree $d_{T}$ in $T$. In order to bound these specializations, we note that for any such $t$ we can bound the size of any root $y\in \mathcal{O}_{K}$ of $P(t,y)=0$. More precisely, Liouville's inequality \eqref{liouville} gives $y\in [cB^{d_{T}}]_{\mathcal{O}_K}$ for some constant $c$ depending on the degree and height of $P$. Thus, the number of $t$'s for which $P(t,Y)$ has an $\mathcal{O}_{K}$-root is contained in the subset of  $(t,y)\in [cB^{d_{T}}]_{\mathcal{O}_{K}}^{2}$. Then, as in \cite{SZ,Walkowiak} we bound the cardinal of this subset by means of a Bombieri-Pila  type of bound. The difference in our proof is that we are going to use a Bombieri-Pila type of bound that we obtained  in \cite{PS}, which is valid for global fields and gives bounds with the $\log(B)$ term removed, and that we are going to use Lemma \ref{serre} which  allows us to deal with polynomials defined in global fields with non-principal ring of integers. 

As in \cite[Definition 3.22]{PS}, if $c_{1}$ is the constant in Lemma \ref{serre} we let
\[\beta:=\begin{cases}27d^{4} & \text{ if char}(K)=0,\\ d^{\frac{14}{3}} & \text{ if }0<\text{char}(K)\leq \max\{27d^{4},c_{1}\},\\ 1 & \text{ if char}(K)>\max\{27d^{4},c_{1}\}.\end{cases}\] 
Given $P\in \mathcal{O}_{K}[T,Y]$ of total degree $d$ we define $b(P):=0$ if $P$ is not absolutely irreducible, otherwise we let
\[b(P):=\prod_{\mathfrak{p}\in \mathcal{P}_{P}}\exp\left(\frac{\log(\mathcal{N}_{K}(\mathfrak{p}))}{\mathcal{N}_{K}(\mathfrak{p})}\right),\]
where
\[\mathcal{P}_{P}:=\{\mathfrak{p}\notin M_{K,\infty}:\mathcal{N}_{K}(\mathfrak{p})>\max\{\beta,c_{1}\}\text{ and }P\text{ mod }\mathfrak{p}\text{ is not absolutely irreducible}\}.\]
When $P$ is absolutely irreducible, \cite[Lemma 3.23]{PS} gives the bound
\begin{equation}
b(P)\lesssim_{K,n}\max\left\{d^{[-2,2]}\log(H_{K,\text{aff}}(P)),1\right\}.
\label{b(f)}
\end{equation}
We are now in position to state the Bombieri-Pila type of bound that we have obtained in \cite{PS}.

\begin{theorem}[{\cite[Corollary 5.15]{PS}}]
Let $K$ be a global field of degree $d_{K}$. Let $P\in \mathcal{O}_{K}[T,Y]$ of total degree $d$ be irreducible, and let $P_{d}$ be its homogeneous part of total degree $d$. Then for any $B\geq 1$, it holds
\begin{align*}
\left|\left\{(t,y)\in [B]_{\mathcal{O}_{K}}^{2}:P(t,y)=0\right\}\right| & \lesssim_{K}B^{\frac{1}{d}}\frac{\min\{d^{[2,6]}\log(H_{K}(P_{d}))+d^{[3,7]}\log(B)+d^{[4,8]},d^{[4,\frac{14}{3}]}b(P)\}}{H_{K}(P_{d})^{\frac{1}{d^2}}}+d\log(B)+d^{[4,8]}\\ & \lesssim_{K}d^{[4,8]}(\log(H_{K,\text{\emph{aff}}}(P))+1)B^{\frac{1}{d}}.
\end{align*}
\label{BP}
\end{theorem}

With this at hand we can now prove Theorem \ref{Hilbert1} of the Introduction.

\begin{proof}[Proof of Theorem \ref{Hilbert1}]
By Lemma \ref{serre}, after multiplying by an adequate non-zero constant, we may suppose that $P$ verifies the bound \eqref{serrebound} with $\lambda=1$. Let us write $P(T,Y)=a_{0}(T)Y^{d_Y}+\cdots +a_{d_Y}(T)$. Let $t\in [B]_{\mathcal{O}_{K}}$ and let $y\in \mathcal{O}_K$ be a root of $P(t,Y)$. By Liouville's inequality \eqref{liouville},  $|y|_v\leq 2\max_i |a_i(t)|_v$ for all $v\in M_{K,\infty}$. Then \eqref{serrebound} implies that 
\begin{equation}
|y|_{v}\leq 2\max_{i}|a_{i}(t)|_{v}\leq 2|d_{T}+1|_{v}^{[1,0]}|P|_{v}|t|_{v}^{d_{T}}\leq 2c_{1}|d_{T}+1|_{v}^{[1,0]}H_{K}(P)^{[\frac{n_{v}}{d_{K}^{2}},\frac{1}{d_K}]}|t|_{v}^{d_{T}}\text{ for all }v\in M_{K,\infty}.
\label{liouville}
\end{equation}
It follows that $y\in [(2c_{1})^{d_{K}^2}(d_{T}+1)^{[d_{K},0]}H_{K}(P)B^{d_{T}}]_{\mathcal{O}_{K}}$. Thus, its enough to bound the number of $2$-tuples $(t,y)$ 
lying in the box $[(2c_{1})^{d_{K}^2}(d_{T}+1)^{[d_{K},0]}H_{K}(P)B^{d_{T}}]_{\mathcal{O}_{K}}^{2}$ which are zeroes of $P(T,Y)$. 
Since Theorem \ref{BP} gives a bound of size $O_{K,P}(B^{\frac{d_{T}}{d}})$,
 we will distinguish between the cases when $d$ is large and when $d$ is small as in \cite{SZ}. To that end, let $H:=\max\{e^{e},H_{K,\text{aff}}(P)\}$, let $L_{1}:=\log(H)$ and let $L_{2}:=\log(\log(H))$. Since $H\geq e^{e}$, it holds $L_{2}\geq 1$.

 Let us suppose first that $d\geq d_{K}^2d_{T}d_{Y}\frac{L_{1}}{L_{2}}$. 
 Then $(d_{T}+1)^{\frac{d_K}{d}}\leq 2$, $H^{\frac{1}{d}}\leq \log(H)^{\frac{1}{d_{K}^2d_{T}d_{Y}}}$ and $B^{\frac{d_T}{d}}\leq B^{\frac{1}{d_{Y}}} $.
By Theorem \ref{BP} applied to  $P$ in the box $[(2c_{1})^{d_{K}^2}(d_{T}+1)^{[d_{K},0]}HB^{d_{T}}]_{\mathcal{O}_{K}}^{2}$ we conclude that the number of $t\in [B]_{\mathcal{O}_{K}}$ for which $P(t,Y)=0$ has a solution in $\mathcal{O}_{K}$ is bounded by 
\begin{equation}
\lesssim_{K}d^{[4,8]}(\log(H_{K,\text{aff}}(P))+1)\left((2c_{1})^{d_{K}^2}(d_{T}+1)^{[d_{K},0]}HB^{d_{T}}\right)^{\frac{1}{d}}\lesssim_{K}d^{[4,8]}((\log(H))^{\frac{1}{d_{K}^2d_{T}d_{Y}}+1}+1)B^{\frac{1}{d_{Y}}}.
\label{bound01}
\end{equation} 

Let us suppose now that  $d<d_{K}^2d_{T}d_{Y}\frac{L_{1}}{L_{2}}$ and let us consider the polynomial $G(T,Y):=P(T,T^{E}+Y)$ where $E=\left\lfloor d_{K}^2d_{T}d_{Y}\frac{L_{1}}{L_{2}}\right\rfloor+1\in [d_{K}^2d_{T}d_{Y}\frac{L_{1}}{L_{2}},2d_{K}^2d_{T}d_{Y}L_{1}]$. Then  
\begin{equation*}
d_{Y}E\leq \deg(G)\leq d_{Y}E+d_{T}\leq 3d_{K}^2d_{Y}^{2}d_{T}L_{1},\text{ and }\log(H_{K,\text{aff}}(G))\lesssim_K [d_Y,1]+\log(H_{K,\text{aff}}(P))\lesssim_K [d_Y,1]+\log(H).
\end{equation*} 
Note that every zero $(t,y)\in \mathcal{O}_{K}^{2}$ of $G$ corresponds to a zero of $P$ of the form $(t,y+t^{E})$. Then \eqref{liouville} implies that any zero $(t,y)$ of $G$ with $t\in [B]_{\mathcal{O}_{K}}$ and $y\in \mathcal{O}_{K}$ verifies
\[|y|_{v}=|-t^{E}+(y+t^E)|_{v}\leq |t|_{v}^{E}+ 2c_{1}|d_{T}+1|_{v}^{[1,0]}H_{K}(P)^{[\frac{n_{v}}{d_{K}^{2}},\frac{1}{d_K}]}|t|_{v}^{d_{T}}\text{ for all }v\in M_{K,\infty}.\]
Then it follows that $y\in [(3c_{1})^{d_{K}^2}(d_{T}+1)^{[d_{K},0]}H_K(P)B^{E}]_{\mathcal{O}_{K}}$. Hence, by Theorem \ref{BP} applied to $G$ in the box $[(3c_{1})^{d_{K}^2}(d_{T}+1)^{[d_{K},0]}HB^{E}]_{\mathcal{O}_{K}}^{2}$ the number of $t\in [B]_{\mathcal{O}_{K}}$ for which $P(t,Y)=0$ has a solution in $\mathcal{O}_{K}$ is bounded by
\begin{align}
& \lesssim_{K}\deg(G)^{[4,8]}\left(\log(H_{K,\text{aff}}(G))+1\right)\left((3c_{1})^{d_{K}}(d_{T}+1)^{[d_{K},0]}HB^{E}\right)^{\frac{1}{\deg(G)}} \nonumber \\  & \lesssim_{K}d_{Y}^{[8,16]}d_{T}^{[4,8]}([d_Y,1]+\log(H))^{[4,8]+1}H^{\frac{L_{2}}{d_{K}^2d_{T}d_{Y}L_{1}}}B^{\frac{1}{d_{Y}}} \nonumber\\ & \lesssim_{K}d_{Y}^{[12,16]}d_{T}^{[4,8]}(\log(H))^{[5,9]} (\log(H))^{\frac{1}{d_{K}^2d_{T}d_{Y}}}B^{\frac{1}{d_{Y}}}\lesssim_{K}d_{Y}^{[12,16]}d_{T}^{[4,8]}(\log(H))^{[5,9]+\frac{1}{d_{K}^2d_{T}d_{Y}}}B^{\frac{1}{d_{Y}}}
\label{bound02}
\end{align} 
Since there are at most $d_T$ values of $t$ for which $a_{0}(t)=0$, from \eqref{bound01} and \eqref{bound02} we deduce the bound stated in Theorem \ref{Hilbert1} written in terms of $H_{K,\text{aff}}(P)$. By Lemma \ref{serre} we recover Theorem \ref{Hilbert1} in terms of $H_{K}(P)$.
\end{proof}

\section{Effective Hilbert's irreducibility theorem}
\label{HIT1}

The goal of this section is to prove Theorem \ref{HIT01} from the Introduction. The idea of the proof is as follows. The irreducibility of a polynomial $F\in K[T,Y]$ implies that for any $t\in \mathcal{O}_{K}$ such that $F(t,Y)$ is reducible over $K$, any non-trivial factorization of $F(t,Y)$ over $K$ gives an irreducible polynomial over $K$ which comes from specialization by $t$ of a non-trivial factor $G(T,Y)$ of $F(T,Y)$ over $\overline{K(T)}$, an algebraic closure of $K(T)$. From the fact that $G$ does not lie in $K[T]$ but $G(t,Y)$ does lie in $K[T]$ it follows that some coefficient of $G$ does not lie in $K[T]$ but its specialization lies in $K$. Then the minimal polynomial $P(T,Y)$ of this coefficient will lie in $K[T,Y]$ and will have a root in $\mathcal{O}_{K}$, thus one is reduced to the problem of bounding the number of specializations of an irreducible polynomial which have $\mathcal{O}_{K}$-roots, and this is exactly the setting of Theorem \ref{Hilbert1} which was proved in Section \ref{section 3}. In order to carry out this strategy, it is required to bound the degree and height of the polynomial $P$. Obtaining such bounds will be the main technical aspect of this section. In order to prove Theorem \ref{HIT01} it will be convenient first to prove a variant of it for polynomials $F(T,Y)\in K[T,Y]$ that are monic on $Y$.  This is the content of the following proposition.

\begin{proposition}
Let $K$ be a global field and let $F\in \mathcal{O}_{K}[T,Y]$ be a polynomial that is monic in $Y$ and irreducible over $K$. Let $d_{T}$ and $d_{Y}$ denote the degrees in $T$ and $Y$ of $F$, respectively. Then 
\[|\left\{t\in [B]_{\mathcal{O}_{K}}:F(t,Y)\text{ is reducible over }K\right\}|\lesssim_{K} 2^{[25d_{Y},36d_{Y}]}d_{T}^{[14,26]}\left(\log(H_{K}(F))+1\right)^{[6,10]}B^{\frac{1}{2}}.\]
\label{Hilbert2}
\end{proposition}

 Before proving Proposition \ref{Hilbert2}, we will require a preliminary lemma, which generalizes \cite[Lemma 3]{SZ} and \cite[Lemme 3.1]{Walkowiak} to global fields. In order to state it, let $\overline{K(T)}$ be an algebraic closure of $K(T)$, and let
\[F(T,Y)=\prod_{i=1}^{d_{Y}}(Y-y_{i})\] 
be the factorization of $F$ over $\overline{K(T)}[Y]$. Since for all $i$, $y_{i}$ is integral over  $\mathcal{O}_{K}[T]$,then for any non-empty subset $\omega\subseteq \{1,\ldots ,d_{Y}\}$ and for any nonnegative integer $j\leq |\omega|$, $\tau_{j}(y_{i}:i\in \omega)$ is integral over $\mathcal{O}_{K}[T]$,  where $\tau_{j}$ is the $j^{\text{th}}$ fundamental symmetric function. 
Hence if $P_{\omega,j}(T,Y)$ denotes the minimal polynomial of $\tau_{j}(y_{i}:i\in \omega)$ over $K(T)$, it follows that $P_{\omega,j}\in \mathcal{O}_{K}[T,Y]$ and it is monic in $Y$. 

\begin{lemma} \label{Pw}
Let $K$ be a global field and let $F\in \mathcal{O}_K[T,Y]$ be a polynomial that is monic in $Y$ and irreducible over $K$. Let $d_T$ and $d_Y$ denote the degrees in $T$ and $Y$ of $F$, respectively. Let us suppose that $d_Y\geq 2$. Let $t\in \mathcal{O}_{K}$ such that  $F(t,Y)$ is reducible over $K$. Then there exists a non-empty subset $\omega\subseteq \{1,\ldots ,d_{Y}\}$ of cardinal $|\omega|\leq \frac{d_{Y}}{2}$  and $j\leq |\omega|$ such that 
\begin{enumerate}
\item $2\leq \deg_{Y}(P_{\omega,j})\leq 2^{d_{Y}}$; \label{c1}
\item $P_{\omega,j}(t,Y)$ has a zero in  $\mathcal{O}_{K}$; \label{c2}
\item $\deg(P_{\omega,j})\leq d_{T}\deg_{Y}(P_{\omega,j})\leq d_{T}2^{d_Y}$; \label{c3}
\item $H_{K}(P_{\omega,j})\lesssim ([2^{d_{Y}+2}(d_{T}+1),1]^{d_K}H_{K}(F))^{\deg_{Y}(P_{\omega,j})}$. \label{c4}
\end{enumerate}
\end{lemma}

\begin{proof}
Since $K[T,y_1,\ldots, y_{d_Y}]$ is integral over $K[T]$, by \cite[Exercise 2, Chap. 5]{MR0242802}, for any $t\in K$  the specialization morphism 
$T\to t$  extends to a morphism from $K[T,y_{1},\ldots, y_{d_Y}]$ to $\overline{K}$. 
 For $y\in K[T,y_{1},\ldots, y_{d_Y}]$ we denote $y(t)$ the image of $y$ under this morphism. 
 
 Let $t$ be as in the statement of the lemma, namely,  $t\in \mathcal{O}_{K}$  such that $F(t,Y)$ is reducible over $K$. 
Since $F$ is monic on $Y$, by Gauss's lemma it is reducible over  $\mathcal{O}_{K}$, so we have a decomposition
\[F(t,Y)=\prod_{i\in \omega}(Y-y_{i}(t))\prod_{i\notin \omega}(Y-y_{i}(t)),\]
where $\omega\subseteq \{1,\ldots, d_{Y}\}$, $|\omega|\leq \frac{d_{Y}}{2}$ and $R(Y):=\prod_{i\in \omega}(Y-y_{i}(t))\in \mathcal{O}_{K}[Y]$. Up to a sign, the coefficients of $R(Y)$ are  equal to $\tau_{j}(y_{i}(t):i\in \omega)$, $j=1,\ldots, |\omega|$. Observe that at least one of the  
  $\tau_{j}(y_{i}:i\in \omega)$ verifies that it is not in $K(T)$, otherwise $F(T,Y)$ would be reducible over $K$.
   Let $\theta_{\omega,j}$ be this element and let $P_{\omega,j}(T,Y)$ be its minimal polynomial over $K(T)$. By the discussion preceding the lemma, $P_{\omega,j}(T,Y)\in \mathcal{O}_{K}[T,Y]$ and is monic in $Y$.
    Note that $\theta_{\omega,j}(t)\in \mathcal{O}_{K}$, hence $P_{\omega,j}$ verifies condition \eqref{c2}.
 
 It is  clear that $\deg_{Y}(P_{\omega,j})\geq 2$, since otherwise $\theta_{\omega,j}$ would lie in $K(T)$. Moreover, the polynomial
 \begin{equation}\label{auxiliar polynomial to estimate degree}
\prod_{\omega'\subseteq \{1,\ldots, d_{Y}\}:|\omega'|=|\omega|}\left(Y-\tau_{j}(y_{i}:i\in \omega')\right)
\end{equation}
has degree in $Y$ equal to ${d_{Y}\choose |\omega|}\leq 2^{d_{Y}}$,  it is invariant under the action of the Galois group of $F$ so it has coefficients in $K(T)$ (and hence in $K[T]$ since $K[T]$ is integrally closed in $K(T)$) and vanishes on $\theta_{\omega,j}$. It follows that $P_{\omega,j}$ divides this polynomial, thus $\deg_{Y}(P_{\omega,j})\leq 2^{d_{Y}}$. This proves condition \eqref{c1}.

Next, we will  bound the height of $P_{\omega,j}$. Let $v\in M_{K,\infty}$. We denote by $\overline{K_{v}}$ the algebraic closure of  the $v$-completion of $K_{v}$. Thus $|\cdot|_v$ extends in a unique way to $\overline{K_v}$. Let us write 
\begin{equation}\label{expansion of Pwj}
P_{\omega,j}=\sum_{i=0}^{\deg_{Y}(P_{\omega,j})}P_{i}(T)Y^{\deg_{Y}(P_{\omega,j})-i} \text{ with }P_{i}(T)=\sum_{h}p_{i,h}T^{h}.
\end{equation}
\begin{claim}
It holds that
\[\max_{h}|p_{i,h}|_{v}\leq \sup_{z\in \overline{K_{v}}:|z|_{v}\leq 1}|p_{i}(z)|_{v}\leq \sup_{z\in \overline{K_{v}}:|z|_{v}\leq 1}|P_{\omega,j}(z,Y)|_{v}\]
\label{bound for the gauss norm}
\end{claim}
\begin{proof}[Proof of Claim \ref{bound for the gauss norm}]
 If $K$ is a number field, $\overline{K_{v}}=\mathbb{C}$, thus by Cauchy's integral formula it holds 
\[|p_{i,h}|_{v}=\left |\frac{P_{i}^{(h)}(0)}{h!}\right |_{v}\leq \sup_{z\in \overline{K_{v}}, |z|_{v}\leq 1}|P_{i}(z)|_{v}\leq \sup_{z\in \overline{K_{v}}:|z|_{v}\leq 1}|P_{\omega,j}(z,Y)|_{v}.\]
 Let us suppose now that $K$ is a function field. By the maximum principle (e.g. see \cite[$\S 2.2$, Proposition 5]{Bosch}), it holds that $\max_{h}|p_{i,h}|_{v}= \sup_{z\in \overline{K_{v}}:|z|_{v}\leq 1}|P_{i}(z)|_{v}$.
\end{proof}

Let $\alpha_{l}(z)$, $1\leq l\leq  \deg_{Y}(P_{\omega,j})$ denote the roots of $P_{\omega,j}(z,Y)$.
By means of expressing the coefficients of  $P_{\omega,j}$ in terms of symmetric functions of its roots, 
with an argument similar to the one in \cite[Lemma 1.6.7]{Bombieri} we will see that 
\begin{equation}
|P_{\omega,j}(z,Y)|_{v}\leq \left|2^{\deg_{Y}(P_{\omega,j})}\right|_{v}\prod_{l=1}^{\deg_{Y}(P_{\omega,j})}\max\{1,|\alpha_{l}(z)|_{v}\}.
\label{bound on P}
\end{equation}
Indeed, if $P_{\omega,j}(z,Y)=\sum_{i=0}^{\deg_Y(P_{\omega,j})}P_i(z)Y^{\deg_Y(P_{\omega,j})-i}$ with $P_0(T)=1$, it holds that $P_i(z)=\sum_{j_1<j_2<\cdots <j_i}\alpha_{j_1}(z)\alpha_{j_2}(z)\cdots \alpha_{j_i}(z)$. Hence, 
\[|P_i(z)|_v\leq \left|{\deg_Y(P_{\omega,j})\choose i}\right|_v\max_{j_1<j_2<\cdots <j_i}|\alpha_{j_1}(z)|_v|\alpha_{j_2}(z)|_v\cdots |\alpha_{j_i}(z)|_v\leq \left|2^{\deg_Y(P_{\omega,j})}\right|_v\prod_{l=1}^{\deg_Y(P_{\omega,j})}\max\{1,|\alpha_{l}(z)|_v\}.\]

 Since $P_{\omega,j}$ divides the polynomial in \eqref{auxiliar polynomial to estimate degree}, each $\alpha_{l}(z)$ is of the form $\tau_{j}(y_{i}(z):i\in \omega')$ for some $\omega'\subseteq \{1,\ldots, d_{Y}\}$ with $|\omega'|=|\omega|$, and hence 
\begin{equation}
|\alpha_{l}(z)|_{v}\leq \left|{ |\omega |\choose j}\right|_{v}\prod_{i\in\omega'}\max\{1,|y_{i}(z)|_{v}\}\leq |2^{|\omega|}|_{v}\prod_{i=1}^{d_{Y}}\max\{1,|y_{i}(z)|_{v}\}.
\label{bound on P2}
\end{equation}

\begin{claim}
Writing $F(T,Y)$ as $F(T,Y)=a_0(T)Y^{d_{Y}}+a_1(T)Y^{d_{Y}-1}+\cdots +a_{d_Y}(T)$ with $a_0(T)=1$, it holds that
\[\prod_{i=1}^{d_{Y}}\max\{1,|y_{i}(z)|_{v}\}\leq \left[\left|d_{Y}+1\right|_v,1\right]\max_{i}|a_{i}(z)|_{v}.\]
\label{bound for the gauss norm2}
\end{claim}
\begin{proof}[Proof of Claim \ref{bound for the gauss norm2}]
Let us suppose that $K$ is a number field. Let $\sigma:K\hookrightarrow \mathbb{C}$ be the embedding corresponding to $v$. We denote by $\sigma(F)$ the usual action of $\sigma$ on $F$.  It follows that $\prod_{i=1}^{d_{Y}}\max\{1,|y_{i}(z)|_{v}\}$ is the Mahler measure of the polynomial $\sigma(F)(z,Y)$. Then \cite[Lemma 1.6.7]{Bombieri} implies the conclusion of the claim.

On the other hand, if $K$ is a function field,  we write $F_{1}(z,Y)=\prod_{i:|y_{i}(z)|_{v}\geq 1}(Y-y_{i}(z))\in \overline{K_{v}}[Y]$ and $F_{2}(z,Y)=\prod_{i:|y_{i}(z)|_{v}<1}(Y-y_{i}(z))\in \overline{K_{v}}[Y]$, thus $F(z,Y)=F_{1}(z,Y)F_{2}(z,Y)$. Observe that $|F_{2}(z,Y)|_{v}=1$, since it is a monic polynomial on $Y$ and its coefficients are symmetric functions on elements of absolute value at most $1$. From this and the fact that $\prod_{i=1}^{d_{Y}}\max\{1,|y_{i}(z)|_{v}\}$ is the absolute value of one of the coefficients of $F_{1}(z,T)$, it follows that
\[\prod_{i=1}^{d_{Y}}\max\{1,|y_{i}(z)|_{v}\}\leq |F_{1}(z,Y)|_{v}=|F_{1}(z,Y)|_{v}|F_{2}(z,Y)|_{v}=|F(z,Y)|_{v}=\max_{i}|a_{i}(z)|_{v},\]
where the identity $|F_{1}(z,Y)|_{v}|F_{2}(z,Y)|_{v}=|F_{1}(z,Y)F_{2}(z,Y)|_{v}$ is Gauss's lemma (e.g. see \cite[Lemma 1.6.3]{Bombieri}).
\end{proof}

By  Claim \ref{bound for the gauss norm}, Claim \ref{bound for the gauss norm2}, inequalities \eqref{bound on P}, \eqref{bound on P2}, and the elementary bound
\begin{equation}\label{bound in the coef of P}
|a_{i}(z)|_{v}\leq |d_{T}+1|_{v}|F|_{v}\max\{1,|z|_{v}^{d_{T}}\} \text{ for all }i,
\end{equation}
we conclude the bound
\begin{equation}
|p_{i,h}|_{v}\leq \left(|2^{|\omega|+1}|_{v}[|d_{Y}+1|_v,1]|d_{T}+1|_{v}|F|_{v}\right)^{\deg_{Y}(P_{\omega,j})}.
\label{bound on P3}
\end{equation}
Since $|\omega|\leq \frac{d_{Y}}{2}$, inequality \eqref{bound on P3} gives the bound on $H_{K}(P_{\omega,j})$ in the item \eqref{c4} statement of the lemma.

It remains to prove \eqref{c3}. To that end, first we will bound the degrees of the $P_{i}$'s in \eqref{expansion of Pwj}. Since for each $z\in K$ $P_{i}(z)$ is up to a sign the $i^{th}$ elementary symmetric function on the roots  $\alpha_{l}(z)$, $1\leq l\leq  \deg_{Y}(P_{\omega,j})$ of $P_{\omega, j}(z,Y)$ 
and these roots can be bounded by combining \eqref{bound on P2}, Claim \ref{bound for the gauss norm2} and \eqref{bound in the coef of P}, then
$$|P_{i}(z)|_{v}\leq \left|{\deg_{Y}(P_{\omega,j})\choose i}\right|_{v}\left([|d_{Y}+1|_v,1] |2^{\omega}|_{v}|d_{T}+1|_{v}|F|_{v}\max\{1,|z|_{v}^{d_{T}}\}\right)^i=O_{F}(|z|_{v}^{d_{T}i}),$$
it follows that $\deg P_{i}\leq d_{T}i$. Then
\[\deg P_{\omega,j}= \max_{0\leq i\leq \deg_{Y}(P_{\omega,j})}(\deg_{Y}(P_{\omega,j})-i+\deg P_i)\leq d_{T}\deg_{Y}(P_{\omega,j}).\]
This ends the proof of the lemma.
\end{proof}
\begin{proof}[Proof of Proposition \ref{Hilbert2}]
We may suppose that $d_Y\geq 2$, otherwise the bound is trivial. Let  $S(B)$ be the number of $t\in [B]_{\mathcal{O}_{K}}$ such that  $F(t,Y)$ is reducible over $K$. Let $S_{\omega}(B)$ be the number of $t\in [B]_{\mathcal{O}_{K}}$  such that $P_{\omega,j}(t,Y)$ has a zero in $\mathcal{O}_{K}$. By Lemma \ref{Pw}\eqref{c2} it holds
\[S(B)\leq \sum_{\omega}S_{\omega}(B)\leq 2^{d_{Y}}\max_{\omega}S_{\omega}(B).\]
Combining Theorem \ref{Hilbert1} and Lemma \ref{Pw} it holds that
\begin{align*}
S_{\omega}(B)& \lesssim_{K} \deg(P_{\omega,j})^{[17,25]}\left( \log H_{K}(P_{\omega,j}) +1\right)^{[6,10]}B^{\frac{1}{\deg_{Y}(P_{\omega,j})}}\\ & \lesssim_{K}2^{[17d_{Y},25d_{Y}]}d_{T}^{[17,25]}\left(2^{d_Y}[(d_{Y}+2)\log(2)+\log(d_{T}+1),0]+2^{d_{Y}}\log(H_{K}(F))+1\right)^{[6,10]}B^{\frac{1}{2}} \\ & \lesssim_{K} 2^{[24d_{Y},35d_{Y}]}d_{T}^{[14,25]}\left(\log(H_{K}(F))+1\right)^{[6,10]}B^{\frac{1}{2}}.
\end{align*}
We conclude that
\begin{equation}
S(B)\lesssim_{K} 2^{[25d_{Y},36d_{Y}]}d_{T}^{[14,25]}\left(\log(H_{K}(F))+1\right)^{[6,10]}B^{\frac{1}{2}}.
\label{S(B)}
\qedhere
\end{equation}
\vspace{-0.3cm}\phantom\qedhere
\end{proof}

Having proved Proposition \ref{Hilbert2}, which is valid for polynomials which are monic in $Y$, we will now proceed to give a proof of a quantitative variant of Hilbert's irreduciblity theorem for global fields, which is the content of Theorem \ref{HIT01} of the Introduction and the main result in this article.

\begin{proof}[Proof of Theorem \ref{HIT01}]
By Lemma \ref{serre}, after multiplying by an adequate non-zero constant, we may suppose that $F$ verifies the bound \eqref{serrebound} with $\lambda=1$. Let us write $F(T,Y)=a_{0}(T)Y^{d_{Y}}+a_{1}(T)Y^{d_{Y}-1}+\cdots +a_{d_{Y}}(T)$ and let
\[G(T,Y):=a_{0}(T)^{d_{Y}-1}F\left(\frac{Y}{a_{0}(T)},T\right)=Y^{d_{Y}}+a_{1}(T)Y^{d_{Y}-1}+a_{0}(T)a_{2}(T)Y^{d_{Y}-2}+\cdots +(a_{0}(T))^{d_{Y}-1}a_{d_{Y}}(T).\]
Then $G$ is monic in $Y$, it belongs to $\mathcal{O}_{K}[T,Y]$,
$\deg_{T}(G)\leq d_Td_Y$, $\deg_{Y}(G)=d_Y$
and if $F$ is irreducible, then $G$ is irreducible. Let us now bound $H_{K}(G)$ in terms of $H_{K}(F)$. To that end, note that for any $v\in M_{K,\infty}$ it holds
\begin{align*}
|G|_{v}& =\max\{1,|a_{1}(T)|_{v},|a_{0}(T)a_{2}(T)|_{v},\ldots ,|a_{0}(T)^{d_{Y}-1}a_{d_{Y}}(T)|_{v}\} \\ & \leq [|2^{d_{T}(\frac{(d_{Y}+1)d_Y}{2}-1)}|_{v},1]\max\{1,|a_{1}(T)|_{v},|a_{0}(T)|_{v}|a_{2}(T)|_{v},\ldots, |a_{0}(T)|_{v}^{d_{Y}-1}|a_{d_{Y}}(T)|_{v}\}\\ & \leq [|2^{d_Td_Y^2}|_{v},1]\max\{1,|F|_{v}^{d_{Y}}\}\leq c_1^{d_Y}[|2^{d_Td_Y^2}|_{v},1]H_K(F)^{[\frac{n_v}{d_K^2},\frac{1}{d_K}]d_Y},
\end{align*}
where in the second inequality we have used that $|fg|_{v}\leq 2^{\deg(f)+\deg(g)}|f|_{v}|g|_{v}$ if $v$ is archimedean (e.g see \cite[Lemma 1.6.11]{Bombieri}) and $|fg|_{v}=|f|_{v}|g|_{v}$ by Gauss's lemma (e.g. see \cite[Lemma 1.6.3]{Bombieri}) if $v$ is non-archimedean, and in the fourth inequality we have used \eqref{serrebound}. Then
$$
H_{K,\text{aff}}(G)=\prod_{v\in M_{K,\infty}}\max\{1,|G|_v\}\leq \begin{cases}   c_1^{d_Yd_K^2}2^{d_{T}d_{Y}^2d_K}H_{K}(F)^{d_{Y}} & \text{ if }K\text{ is a number field}, \\ c_1^{d_Yd_K}H_{K}(F)^{d_{Y}} & \text{ if }K\text{ is a function field}.
\end{cases}
$$
By Proposition \ref{Hilbert2} it holds that the number of $t\in [B]_{\mathcal{O}_{K}}$ with $a_{0}(t)\neq 0$ and $F(t,Y)$ reducible over $K$ is bounded by
\begin{align}
\left|\left\{ t\in [B]_{\mathcal{O}_{K}}:G(t,Y) \text{ is reducible over }K \right\}\right| &  \lesssim_{K} 2^{[25\deg_{Y}(G),36\deg_{Y}(G)]}\deg_{T}(G)^{[14,25]}\left(\log(H_{K,\text{aff}}(G))+1\right)^{[6,10]}B^{\frac{1}{2}} \label{t diff of 0} \\ & \lesssim_{K} 2^{[25d_{Y},36d_{Y}]}d_{Y}^{[26,35]}d_{T}^{[20,25]}\left(\log(H_{K}(F))+1\right)^{[6,10]}B^{\frac{1}{2}}. \nonumber
\end{align}
Taking into account that there are at most $d_{T}$ values of $t$ for which $a_{0}(t)=0$ we see that \eqref{t diff of 0} implies the bound in the statement of Theorem  \ref{HIT01} .\end{proof}

\begin{remark}
In some cases by the same method of the proof of Theorem \ref{HIT01} one can derive a polynomial bound in $d_Y$ and $d_T$. Specifically,  if the polynomial $F$ in Theorem \ref{HIT01} is also assumed to be separable, combining our techniques with \cite[$\S$ 3]{MR2441693} would yield the bound
\[|\left\{t\in [B]_{\mathcal{O}_{K}}:F(t,Y)\text{ is reducible over }K\right\}| 
\lesssim_{K} d_{Y}^{c(F)}d_{T}^{[21,26]}\left(\log(H_{K}(F))+1\right)^{[6,10]}B^{\frac{1}{2}},\]
where $c(F)$ is, modulo an absolute constant, the Hilbert index defined in \cite[$\S$ 3.2]{MR2441693}, which in some cases, e.g. when the Galois group of $F$ is bounded by $d_Y^{\alpha}$ for some $\alpha>0$, can be bounded by an absolute constant (see \cite[Theorem 4.1]{MR2441693}).
\label{polyhilbertremark}
\end{remark}

From Theorem \ref{HIT01} it follows easily the next quantitative variant of Hilbert's irreducibility theorem which improves upon \cite[Theorem]{SZ} and \cite[Th\'eor\`eme 3]{Walkowiak} when $K=\mathbb{Q}$.

\begin{corollary}
Let $s\in \mathbb{N}$ and let $F(T,Y)\in \mathcal{O}_K[T,Y]$ be an irreducible polynomial of degrees $d_T$ in $T$ and $d_Y$ in $Y$. Let $C:=2^{[25d_{Y},36d_{Y}]}d_{Y}^{[26,35]}d_{T}^{[21,26]}\left(\log(H_{K}(F))+1\right)^{[6,10]}$. Then there exist a constant $c_K\lesssim_K 1$ and $s$ algebraic integers $t_1,\ldots ,t_s\in [s+c_KC^2]_{\mathcal{O}_K}$ such that the polynomials $F(t_i,Y)$ are irreducible over $K$ for all $i=1,\ldots, s$.
\end{corollary}

\section{Distribution of Galois groups}
The goal of this section is to prove Theorem \ref{Hilbert3.5}, Theorem \ref{Hilbert7}, and Theorem \ref{HIT3} of the Introduction. As with many instances of Hilbert's irreducibility theorem, it is possible to adapt the strategy of Section \ref{HIT1} to bound the subset of specializations that have a Galois group different from the Galois group of a given irreducible polynomial $F\in K[T,Y]$. More precisely, let $F\in K[T,Y]$ be an irreducible polynomial, monic in $Y$. Let us denote by $L$ the splitting field of the polynomial $F$ over $K(T)$. Let $G$ be the Galois group of $L/K(T)$. Given $t\in K$ we let $L_{t}$ be the splitting field of the polynomial $F(t,Y)\in K[Y]$. Let $G_{t}$ be the Galois group of $L_t/K$. Let us suppose that $F$ is separable as a polynomial in $K(T)[Y]$ and let $\Delta(T)$ be the discriminant of $F$, which we regard as a polynomial in $K[T]$. Let $A=K[T]$ and let $C$ be the integral closure of $A$ in $L$. Let $\mathfrak{p}=(T-t)K[T]$ and let $\mathfrak{P}$ be a prime in $C$ lying over $\mathfrak{p}$.  Let $t\in K$ such that $\Delta(t)\neq 0$, i.e. the polynomial $F(t,Y)$ does not have multiple roots in $K=A/\mathfrak{p}$. Since $L_{t}=C/\mathfrak{P}$ and the Galois group $G_t$ of $L_t/K$ is naturally identified with the Galois group of $(C/\mathfrak{P})/(A/\mathfrak{p})$, by \cite[Chapter VII, Section 2, Theorem 2.9]{LangA} there is an isomorphism of groups $\iota: G_t\to G_{\mathfrak{P}}$, where $G_{\mathfrak{P}}$ is the decomposition group of $\mathfrak{P}$, which is a subgroup of $G$. If $\rho$ denotes the natural embedding of $G$ in the permutation group $S_{\deg_Y(F)}$, it follows that there are injective morphisms $\iota,\rho_t$ such that the following is a commutative diagram
\begin{equation*}
\begin{tikzcd}
G_{t}\arrow[r,"\iota"]\arrow[rd, "\rho_{\boldsymbol t}"] & G\arrow[d, "\rho"]\\ 
& S_{\deg_Y(F)}
\end{tikzcd}
\end{equation*}
With all this at hand we are able to prove Theorem \ref{Hilbert3.5} and Theorem \ref{Hilbert7} of the Introduction.

\begin{proof}[Proof of Theorem \ref{Hilbert3.5}]
The strategy of the proof of Theorem \ref{Hilbert3.5} consists of making effective \cite[Chapter 9, Section 9.2, Proposition 2]{Serre} and then using  Theorem \ref{HIT01}. More precisely, let us suppose that $F$ is monic on $Y$.  By Lemma \ref{serre}, after multiplying by an adequate non-zero constant, we may suppose that $F$ verifies the bound \eqref{serrebound} with $\lambda=1$. Since $L/K(T)$ is a Galois extension, all its subextensions are of the form $L^{H}$ for some subgroup $H\leq G$. By the separability assumption, they are of the form $K(T)(\alpha_H)$ for some $\alpha_H\in L$, which we may suppose it is integral over $K[T]$. Let $P_H\in K[T,Y]$ be an irreducible polynomial of minimal degree which has $\alpha_H$ as a root.

\begin{claim}
Let $\mathcal{H}$ be the subset of subgroups of $G$. It holds that 
\[\{t\in [B]_{\mathcal{O}_K}:G_t\neq G\text{ and }\Delta(t)\neq 0\}\subseteq \bigcup_{H\in \mathcal{H}}\{t\in [B]_{\mathcal{O}_K}:P_H(t,Y)\text{ is reducible over }K\}.\]
\label{reduction to irreducibility}
\end{claim}
\begin{proof}[Proof of Claim \ref{reduction to irreducibility}]
Let $A=K[T]$ and let $C$ be the integral closure of $A$ in $L$. Let $t\in [B]_{\mathcal{O}_K}$ such that $\Delta(t)\neq 0$. Let $\mathfrak{P}$ be any prime lying over $\mathfrak{p}=(T-t)A$. Since $F$ is monic in $Y$, $G_t$ is identified with the decomposition group of $\mathfrak{P}$. Let us assume that $G_t\neq G$. Then there are $\sigma\in G$, $f\in \mathfrak{P}\subseteq C$ such that $\sigma(f)\notin \mathfrak{P}$. Moreover, there is $H\in \mathcal{H}$ such that $f\in K(T)(\alpha_H)$. Since $f$ is integral over $A$, it lies in the integral closure of $A$ over $K(T)(\alpha_H)$. Since  $\alpha_H$ is integral and separable, by the proof of \cite[Proposition 13.14]{Eisenbud} it follows that $f\in K[T][\{\frac{\alpha_H^i}{\Delta(T)^2}\}_{i}]$. Then, by cleaning denominators, there are $N\in \mathbb{N}_{0}$ and $P\in K[T,Y]$ such that $\Delta(T)^Nf=P(T,\alpha_H)$. By assumption, $f\in \mathfrak{P}$, thus $P(t,\alpha_H(t))=0$. Arguing by contradiction, let us suppose that $P_H(t,Y)$ is irreducible over $K$. Since $P_H(t,\alpha_H(t))=0$ it follows that  there is some $Q(Y)\in K[Y]$ verifying $P(t,Y)=Q(Y)P_H(t,Y)$. Since $P_H(T,Y)$ has as a root the element $\sigma(\alpha_H)$, it follows that $P_H(t,\sigma(\alpha_H)(t))=0$. Then $P(t,\sigma(\alpha_H)(t))=0$. But $\Delta(T)^N f=P(T,\alpha_H)$ implies that $\sigma(\Delta(T)^N f)=\Delta(T)^N \sigma(f)=P(T,\sigma(\alpha_H))$. The assumption that $\sigma(f)(t)\notin \mathfrak{P}$ implies that $\sigma(f)(t)\neq 0$. This, together with the fact that $\Delta(t)\neq 0$, implies that $P(t,\sigma(\alpha_H)(t))\neq 0$, which is a contradiction.  
\end{proof}
By Claim \ref{reduction to irreducibility}, Theorem \ref{Hilbert3.5} will follow from Theorem \ref{HIT01} if one can bound the heights and the degrees of the polynomials $P_H(T,Y)$. For this purpose, we will require to choose an adequate $\alpha_H$ for each $H\in \mathcal{H}$. This will be achieved in the next two claims.
 
\begin{claim}
Let $F(T,Y)=\prod_{i=1}^{d_Y}(Y-y_i)$ with $y_i$ lying in an algebraic closure of $K(T)$. There exists $\alpha_G=y_1+\gamma_1y_2+\cdots +\gamma_{d_Y-1}y_{d_Y}$ with $\gamma_i\in [d_Y^{i+1}]_{\mathcal{O}_{\mathbbm{k}}}$ for all $1\leq i\leq d_{Y}-1$, such that $L=K(T)(\alpha_G)$.  
\label{bound for the minimal polynomial}
\end{claim}
\begin{proof}[Proof of Claim \ref{bound for the minimal polynomial}]
Recall that $\mathbbm{k}=\mathbb{Q}$ if $K$ is a number field, and $\mathbbm{k}=\mathbb{F}_{q}[T]$ if $K$ is a function field. By examining the proof of the primitive element theorem as in \cite[Theorem 5.1]{Milne} we see that $y_1+\gamma y_2$ is a primitive element for $K(T)(y_1,y_2)$ for any $\gamma\in L$ which is not of the form $\frac{y_i-y_1}{y_2-y_j},j\neq 2$. Since there are at most $d_Y^2$ such possible elements in $L$ and $[d_Y^2]_{\mathcal{O}_{\mathbbm{k}}}$ has more that $d_Y^2$ elements, it follows that we may take $\gamma\in [d_Y^2]_{\mathcal{O}_{\mathbbm{k}}}$. Arguing by induction it follows that there are $\gamma_1,\ldots \gamma_{d_Y-1}$ such that $\alpha_G=y_1+\gamma_1 y_2+\ldots +\gamma_{d_Y-1}y_{d_Y}$ is a primitive element for $L/K(T)$, and $\gamma_i\in [d_{Y}^{i+1}]_{\mathcal{O}_{\mathbbm{k}}}$ for all $i$.
\end{proof}

\begin{claim}
There exists $\alpha_H\in L$ integral over $K[T]$  for which there exists a polynomial  $P_H(T,Y)\in \mathcal{O}_K[T,Y]$ that is a multiple of the minimal polynomial of $\alpha_H$ over $K(T)$, such that $\deg(P_H)\leq d_T|H||G|$ and
\[H_{K,\emph{\text{aff}}}(P_H)\leq c_1^{d_K|H||G|+d_K^2}2^{(|G|+|H||G|+|H||G|d_K+d_T|H||G|)d_K}|G|^{(|H|+1)|G|d_K}(d_Y^{d_Y+1})^{|H||G|d_K}(d_T+1)^{|H||G|d_K}H_{K}(F)^{|H||G|}.\]
\label{bound for the intermediary fields}
\end{claim}

\begin{proof}[Proof of Claim \ref{bound for the intermediary fields}]
It holds that $L^H$ is the field extension of $K(T)$ generated by the coefficients of the polynomial $\prod_{\rho\in H}(Y-\rho(\alpha_G))$, which are of the form $\tau_j(\rho(\alpha_G):\rho\in H),1\leq j\leq |H|$. Note that the coefficients of this polynomial are integral over $K[T]$, since by construction $\alpha_G$ is integral over $K[T]$. Let $\tau_j:=\tau_j(\rho(\alpha_G):\rho\in H),1\leq j\leq |H|$. Since all the $\tau_j$ lie in $L$, they are elements of degree at most $|G|$. Then repeating the argument of Claim \ref{bound for the minimal polynomial} we may take $\alpha_H=\tau_1+\delta_1\tau_2+\cdots +\delta_{|H|-1}\tau_{|H|}$ with $\delta_i\in [|G|^{i+1}]_{\mathcal{O}_{\mathbbm{k}}}$ for all $1\leq i\leq |H|-1$. 

Let $P_H(T,Y)$ be the minimal polynomial of $\alpha_H$ over $K[T]$. By construction of $\alpha_H$, it is monic in $Y$, thus $H_K(P_H)=H_{K,\text{aff}}(P_H)$, and $\deg_Y(P_H(T,Y))=|G/H|$. Let $P(T,Y):=\prod_{\rho\in G}(Y-\rho(\alpha_H))$. Since $P$ is invariant under the action of the Galois group $G$ and $\alpha_H$ is integral over $K[T]$, it holds that $P(T,Y)\in K[T,Y]$. Furthermore, $P_H(T,Y)$ divides $P(T,Y)$ in $K(T)[Y]$ because it has $\alpha_H$ as a root. Since the principal term of $P_H(T,Y)$ as a polynomial in $Y$ is a scalar, we conclude that there is some $Q(T,Y)\in K[T,Y]$ such that $P(T,Y)=Q(T,Y)P_H(T,Y)$. Then, for all $v\in M_K$ it follows that
\begin{equation}
|P(T,Y)|_v\begin{cases}=|Q(T,Y)|_v|P_H(T,Y)|_v & \text{ if }v\text{ is non-archimedean}\\ \geq \left|2^{-(\deg(P(T,Y))}\right|_v|Q(T,Y)|_v |P_H(T,Y)|_v & \text{ if }v\text{ is archimedean}\end{cases},
\label{lower bound for the absolute value}
\end{equation}
where we have used Gauss's lemma $|fg|_v=|f|_v|g|_v$ if $v$ is  non-archimedean, and the inequality $|fg|_v\geq 2^{-(\deg(f)+\deg(g))}|f|_v|g|_v$ if $v$ is archimedean (e.g. see \cite[Lemma 1.6.11]{Bombieri}). From \eqref{lower bound for the absolute value}  and the fact that the height of any non-zero polynomial is at least $1$, it follows that 
\begin{align*}
H_{K,\text{aff}}(P_H(T,Y))=H_K(P_H(T,Y))\leq H_K(P_H(T,Y))H_K(Q(T,Y)) & \leq 2^{\deg(P(T,Y))d_K}H_K(P(T,Y))\\ & \leq 2^{\deg(P(T,Y))d_K}H_{K,\text{aff}}(P(T,Y)).
\end{align*}

Thus, in order to bound $H_{K,\text{aff}}(P_H(T,Y))$ it is enough to bound the height and degree of $P(T,Y)$. 

Let us first bound the height of $P(T,Y)$. Note that for any $\rho\in G$, $\rho(\tau_j)=\tau_j((\rho\sigma)(\alpha_G):\sigma\in H),1\leq j\leq |H|$. Furthermore, since $G$ acts by permutation on the roots of $F$, from the construction of $\alpha_G$ it follows that for all $\rho\in G$, for all $v\in M_{K,\infty}$, and for all $z\in \overline{K}_v$ it holds the bound
\begin{equation}
|\rho(\alpha_G)(z)|_v=\left|y_{\rho(1)}(z)+\gamma_1y_{\rho(2)}(z)+\cdots +\gamma_{d_Y-1}y_{\rho(d_Y)}(z)\right|_v\leq |d_Y^{d_Y+1}|_v\max_{1\leq i\leq d_Y}\{|y_i(z)|_v\}.
\label{primitive element bound}
\end{equation}
Then, for all $\rho\in G$, for all $v\in M_{K,\infty}$, and for all $z\in \overline{K}_v$ it follows that
\[|\rho(\tau_j)(z)|_v\leq \left|{|H|\choose j}\right|_v\max_{I\subseteq H:|I|=j}\prod_{\sigma \in I}|(\rho\sigma)(\alpha_G)(z)|_v\leq |2^{|H|}|_v|d_Y^{d_Y+1}|_v^j\max_i\{|y_i(z)|_v^{j}\}\text{ for all } 1\leq j\leq |H|,\]
and hence 
\begin{align}|\rho(\alpha_H)(z)|_v =\left|\rho(\tau_1)(z)+\delta_1\rho(\tau_2)(z)+\cdots +\delta_{|H|-1}\rho(\tau_{|H|})(z)\right|_v & \leq \left||H||G|^{|H|}\right|_v\max_j\{|\rho(\tau_j)(z)|_v\} \nonumber \\  & \leq \left||G|^{|H|+1}\right|_v|2^{|H|}|_v|d_Y^{d_Y+1}|_v^{|H|}\max\{1,|y_i(z)|_v\}^{|H|}.
\label{bound for the conjugate roots of alphaH}
\end{align}

By Liouville's inequality \eqref{liouville}, $|y_i(z)|_v\leq 2\max_{1\leq i\leq d_Y}\{|a_i(z)|_v\}$. Then, arguing as in the proof of inequality \eqref{bound on P}, from inequalities \eqref{bound in the coef of P} and  \eqref{bound for the conjugate roots of alphaH}, it follows that for all $v\in M_{K,\infty}$ and for all $z\in \overline{K_v}$ with $|z|_v\leq 1$ it holds the bound
\begin{align}
|P(z,Y)|_v & \leq \left| 2^{|G|}\right|_v \prod_{\rho\in G}\max\{1,|\rho(\alpha_H)(z)|_v\}\leq \left|2^{|G|}\right|_v\left(\left||G|^{|H|+1}\right|_v|2^{|H|}|_v|d_Y^{d_Y+1}|_v^{|H|}\right)^{|G|}  \left(\max_{1\leq i\leq d_Y}\{1,|y_i(z)|_v\}\right)^{|H||G|} \nonumber\\ & \leq \left|2^{|G|}\right|_v\left(\left||G|^{|H|+1}\right|_v|2^{|H|}|_v|d_Y^{d_Y+1}|_v^{|H|}\right)^{|G|} 2^{|H||G|} \left(\max_{1\leq i\leq d_Y}\{1,|a_i(z)|_v\}\right)^{|H||G|}\nonumber \\ & \leq \left|2^{|G|}\right|_v\left(\left||G|^{|H|+1}\right|_v|2^{|H|}|_v|d_Y^{d_Y+1}|_v^{|H|}\right)^{|G|} 2^{|H||G|} |d_T+1|_v^{|H||G|} \max\{1,|F|_v\}^{|H||G|}.
\label{primitive element bound2}
\end{align}
By the same argument as in the proof of Claim \ref{bound for the gauss norm},  and by inequality \eqref{primitive element bound2}, we conclude that for all $v\in M_{K,\infty}$,
\begin{equation}
|P(T,Y)|_v\leq \left|2^{|G|}\right|_v\left(\left||G|^{|H|+1}\right|_v|2^{|H|}|_v|d_Y^{d_Y+1}|_v^{|H|}\right)^{|G|} 2^{|H||G|} |d_T+1|_v^{|H||G|}\max\{1,|F|_v\}^{|H||G|}.
\label{bound11}
\end{equation}
Then
\begin{align}
H_{K,\text{aff}}(P(T,Y)) & =\left(\prod_{v\in M_{K,\infty}}\max\{1,|P(T,Y)|_v\}\right)^{d_K}  \nonumber\\ & \leq \left( \prod_{v\in M_{K,\infty}} (2c_1)^{|H||G|}\left|2^{|G|+|H||G|}|G|^{(|H|+1)|G|}(d_Y^{(d_Y+1)|H||G|})(d_T+1)^{|H||G|}\right|_vH_K(F)^{[\frac{n_v}{d_K^2},\frac{1}{d_K}]|H||G|} \right)^{d_K} \nonumber \\ & =c_1^{d_K|H||G|}2^{(|G|+|H||G|+|H||G|d_K)d_K}|G|^{(|H|+1)|G|d_K}(d_Y^{d_Y+1})^{|H||G|d_K}(d_T+1)^{|H||G|d_K}H_{K}(F)^{|H||G|}.
\end{align}

Let us now  bound the degree of $P(T,Y)$. Note that the coefficients of $P(z,Y)$ as a polynomial in $K[Y]$ are, up to a sign, symmetric functions on the roots $\rho(\alpha_H),\rho\in G$. Then, arguing as in the proof of Lemma \ref{Pw}\eqref{c3}, it follows that $\deg(P)\leq d_T|H||G|$. 
\end{proof}

By Theorem \ref{HIT01} and Claim \ref{bound for the intermediary fields}, it follows that
\begin{align}
|\{t\in [B]_{\mathcal{O}_K}:P_H(t,Y & ) \text{ is reducible over }K\}| \nonumber\\ & \lesssim_K 2^{[25\deg_Y(P_H),36\deg_Y(P_H)]}\deg_Y(P_H)^{[26,35]}\deg_T(P_H)^{[21,26]}(\log(H_{K,\text{aff}}(P_H))+1)^{[6,10]}B^{\frac{1}{2}} \nonumber \\ & \lesssim_K 2^{[25|G/H|,36|G/H|]}|G/H|^{[26,35]}d_T^{[28,37]}d_Y^{[7,11]}|H|^{[27,36]}|G|^{[28,37]}(\log(H_{K}(F))+1)^{[6,10]}B^{\frac{1}{2}}. \label{Hilbertbound00}
\end{align}
Theorem \ref{Hilbert3.5} follows from Claim \ref{reduction to irreducibility}, inequality \eqref{Hilbertbound00}, and by taking into account that $|\mathcal{H}|\lesssim_{|G|}1\lesssim_{d_Y}1$ and that the contribution of the $t$'s with $\Delta(t)=0$ only changes the implicit constant. 

The case when $F$ is non monic is dealt as in the proof of Theorem \ref{HIT01}.
\end{proof}

\begin{remark}
In some cases  the argument in the proof of Theorem \ref{Hilbert3.5} can give polynomial dependence on $d_Y$ and $d_T$. More specifically, the bounds in $d_Y$ in general are not polynomial because both proofs reduce the problem of bounding the exceptional specializations to that of bounding the number of specializations of some polynomials that depend on the Galois group $G$ of $L/K(T)$. In general, the group $G$ can be as large as $d_Y!$. Nonetheless, if $|G|\leq d_Y^{\alpha}$ for some absolute constant $\alpha$,  the degree of the minimal polynomial $P_{H}$ can be bounded polynomially on $d_Y,d_T$. Furthermore, instead of considering $\mathcal{H}$ as all subgroups of $G$, one may just consider those maximal subgroups of $G$. In this case, there are results that give polynomial bounds for $\mathcal{H}$, e.g. by \cite[Theorem 1.3]{MR2360145}, $|\mathcal{H}|\lesssim |G|^{\frac{3}{2}}$. Taking these observations in consideration, by Remark \ref{polyhilbertremark} one can prove Theorem \ref{Hilbert3.5} with a polynomial bound in $d_Y$.
\end{remark}

We now give a proof of Theorem \ref{Hilbert7} of the Introduction.

\begin{proof}[Proof of Theorem \ref{Hilbert7}]
As in the proof of Theorem \ref{HIT01} we may suppose that $F$ is monic, with $F(T,Y)=Y^{d_Y}+g_{1}(T)Y^{d_Y-1}+\cdots +g_{d_Y}(T)\in \mathcal{O}_{K}[T,Y]$. It is easy to see that \cite[Lemma 4]{Castillo} is also valid for number fields, so let $\Phi_{F,K}(Z,T)$ be the Galois resolvent constructed in that lemma. Then, if $t\in [B]_{\mathcal{O}_K}$ is such that $F(t,Y)$ has Galois group $H$ over $K$, then $\Phi_{F,H}(Z,t)$ has a root $z\in \mathcal{O}_{K}$. Factoring $\Phi_{F,H}(Z,T)$ over $K$, by Gauss lemma each irreducible factor can be assumed to have $\mathcal{O}_{K}$-coefficients, monic on $Z$, and having degree at least $|G/H|$ and at most $|S_{d_Y}/H|$. Moreover, there are at most $\frac{|H|}{|G|}|S_{d_Y}|$ such factors. We conclude the proof by applying Theorem \ref{Hilbert1} to each such irreducible factor of $\Phi_{F,H}(Z,T)$. 
\end{proof}

As a corollary, we deduce Theorem \ref{HIT3} of the Introduction.

\begin{proof}[Proof of Theorem \ref{HIT3}]
The bounds follow from Theorem \ref{Hilbert7} and from the fact that if $F(T,Y)=g_{0}(T)Y^{d_Y}+g_{1}(T)X^{d_Y-1}+\cdots +g_{d_Y}( T)$, then $g_{0}(T)\neq 0$ and $\Delta(T)\neq 0$ and the number of exceptional $ t\in [B]_{\mathcal{O}_{K}}$ for which $f(t,Y)$ has degree less than $d_Y$ or becomes inseperable are those $t$ for which $g_{0}(t)=0$ or $\Delta(t)=0$ is bounded by $\lesssim_{K,\deg(F)}1$. 

\end{proof}

As a concluding remark, we mention that, arguing by induction on the number of parameters and using Kronecker's substitution (see for instance \cite[Lemma 7.1]{Cohen}), one may extend Theorem \ref{Hilbert1} and then Theorem \ref{HIT01}, Theorem \ref{Hilbert3.5}, Theorem \ref{Hilbert7} and Theorem \ref{HIT3} to polynomials $F\in K[T_1,\ldots, T_s,Y]$. However, by the nature of the induction and the fact that Theorem \ref{Hilbert1} depends on $H_K(F)$, the bound has a factor $(\log(B))^{\nu}$ for some $\nu$. Thus, in the higher dimensional case the bounds one would obtain are not optimal in $B$.  For the sake of completeness we state without the proof the bound that would be obtained by this procedure that extends Theorem \ref{Hilbert1} to the case of several parameters.

\begin{theorem}
Let $K$ be a global field. Let $P(T_{1},\ldots, T_{s},Y)\in \mathcal{O}_{K}[T_{1},\ldots, T_{s},Y]$ be irreducible of degree $d_Y$ in $Y$, and let us suppose that $P$ is monic of degree $d_{Y}$. Then there are positive constants $\mu,\nu\lesssim_{s}1$ such that
\[\left|\left\{ (t_{1},\ldots, t_{s})\in [B]_{\mathcal{O}_{K}}^{s}:P(t_{1},\ldots, t_{s},Y)=0\text{ has a solution }y\in \mathcal{O}_{K} \right\}\right|\lesssim_{K,s}(\log(H_{K}(P))+1)^{\mu}B^{s-1+\frac{1}{d_{Y}}}(\log(B))^{\nu}.\]
\end{theorem}

\bibliography{paper}
\bibliographystyle{abbrv}

\end{document}